\documentclass[11pt]{article}
\usepackage{epsfig}
\usepackage{amsmath,amssymb,amsthm}
\usepackage[utf8]{inputenc}

\begin{document}
%
\newtheorem{theorem}{Theorem}
\newtheorem{prop}{Proposition}
\newtheorem{proposition}{Proposition}
\newtheorem{lemma}{Lemma}

\newcommand{\beqn}{\begin{equation}}
\newcommand{\eeqn}{\end{equation}}
\def\R{{\mathbb R}}
\def\D{{\cal D}}
\newcommand{\cleq}{\preccurlyeq}
\newcommand{\nn}{\nonumber}
\def\la{{\langle}}
\def\ra{{\rangle}}
\def\pa{{\partial}}
\def\ep{\epsilon}
\def\F{{\cal F}}
\def\oml{{\ddot{\text{o}}}}

\title{ Uniqueness for the inverse backscattering problem for angularly controlled potentials}

\author{
Rakesh \\
Department of Mathematical Sciences\\
University of Delaware\\
Newark, DE 19716, USA\\
~\\
Email: rakesh@math.udel.edu\\
\and
Gunther Uhlmann\\
Department of Mathematics\\
University of Washington\\
Seattle, WA 98195, USA\\
~\\
Email: gunther@math.washington.edu
 }

\date{February 12, 2014}
\maketitle

\begin{abstract}
We consider the problem of recovering a smooth, compactly supported potential on $\R^3$ from its backscattering data. We show that if two such potentials have the same backscattering data and the difference of the two potentials has controlled angular derivatives then the two potentials are identical. In particular, if two potentials differ by a finite linear combination of spherical harmonics with radial coefficients and have the same backscattering data then the two potentials are identical.
\end{abstract}




\section{Introduction}

\subsection{Goal}
 
 Let $B$ denote the closed unit ball in $\R^3$, $S$ denote the unit sphere in $\R^3$ and suppose $q(x)$ is a smooth real valued 
 function on $\R^3$ with support in $B$.  Given a unit vector $\omega$ in $\R^3$, let $U(x,t, \omega)$ be the solution of the IVP(Initial 
 Value Problem)
\begin{gather}
U_{tt} - \Delta U + qU = 0, \qquad (x,t) \in \R^3 \times \R,
\label{eq:Ude}
\\
U(x,t) = \delta( t - x \cdot \omega), \qquad x \in \R^3, ~ t< -1;
\label{eq:Uic}
\end{gather}
here $x \cdot \omega$ denotes the inner product of $x$ and $\omega$ and $\delta( \cdot)$ is the Dirac delta distribution.
We may express $U(x,t,\omega)$ in the form
\[
U(x,t, \omega) = \delta( t - x \cdot \omega) + u(x,t,\omega)
\]
where $u(x,t,\omega)$ is the solution of the IVP
\begin{gather}
u_{tt} - \Delta u + qu = - q \, \delta(t - x \cdot \omega), \qquad (x,t) \in \R^3 \times \R
\label{eq:utempde}
\\
u(x,t) = 0, \qquad x \in \R^3, ~t< -1.
\label{eq:utempic}
\end{gather}

If we regard $q(x)$ as representing some physical property of a medium occupying $\R^3$ then $u(x,t,\omega)$ may be 
regarded as the response of the medium to an incoming plane wave $\delta(t- x \cdot \omega)$ moving in the direction $\omega$. The ``far field 
pattern'' (defined carefully later) of the medium response, measured in the direction of the unit vector $\theta$ in $\R^3$,
 with delay $s \in \R$ is 
\[
\lim_{r \to \infty} r u(r \theta, r - s, \omega).
\]
A longstanding open problem is the recovery of the the medium property $q$ from the backscattered far field data (which 
consists of far field pattern measured only in the direction $\theta=-\omega$, for all incoming wave directions $\omega$ 
and for all delays $s \in \R$). The problem is equivalent to the inversion of the map from $q$ to the backscattered data 
and an important first step is to prove the injectivity of this map, which itself is a long standing open problem. Our main result 
is that if the backscattered data for two potentials coincide and the difference of the two potentials has controlled angular 
derivatives then the two potentials are identical. 

\subsection{The problems and the results}

We first define what we mean by the far field pattern with the help of the following theorem.
\begin{theorem}[Properties of the forward map]\label{thm:forward} 
Suppose $q(x)$ is a smooth function on $\R^3$ which is supported in the unit ball $B$ and $\omega, \theta$ are arbitrary unit vectors in $\R^3$.
\vspace{-0.2in}
\begin{enumerate}
\item[(a)] $u(x,t,\omega)$ is supported in the region $ t \geq x \cdot \omega$ and, in this region, $u$ is the unique smooth solution of the characteristic IVP problem 
\begin{gather}
u_{tt} - \Delta u + qu = 0, \qquad (x,t) \in \R^3 \times \R, ~~t \geq x \cdot \omega,
\label{eq:ude}
\\
u(x, x \cdot \omega, \omega) = - \frac{1}{2} \int_{-\infty}^0 q(x + \sigma \omega) \, d \sigma,
\qquad x \in \R^3
\label{eq:ubc}
\\
u(x,t, \omega) = 0, \qquad x \in \R^3, ~t< -1.
\label{eq:uic}
\end{gather}
\item[(b)] As distributions in $s \in \R$ we have
\beqn
\lim_{r \rightarrow \infty}  r \, u(r \theta, r-s, \omega) =   - \frac{1}{2 \pi}
\int_{x \cdot \theta = 1} (\theta \cdot \nabla u)(x,1 -s,\omega) \, dS_x.
\label{eq:farfield}
\eeqn
\item[(c)] Further, as distributions in $t \in \R$, we have
\beqn
\pa_t \int_{x \cdot \theta = 1} u(x, t, \omega) \, dS_x
= - \int_{x \cdot \theta = 1} (\theta \cdot \nabla u)(x,t,\omega) \, dS_x ,
\label{eq:DN}
\eeqn
and
\beqn
\int_{x \cdot \theta=\tau} u(x,t, \omega) \, dS_x = \int_{x \cdot \theta=1} u(x,t-\tau+1,\omega) \, dS_x,
\qquad \text{for all} ~ \tau \geq 1.
\label{eq:transport}
\eeqn
\end{enumerate}
\end{theorem}
Actually our proof of (\ref{eq:farfield}) shows something stronger; if $x^* \in \R^3$ is orthogonal to $\theta$ then
\[
lim_{r \rightarrow \infty}  r \, u(x^*+r \theta, r-s, \omega) =   - \frac{1}{2 \pi}
\int_{x \cdot \theta = 1} (\theta \cdot \nabla u)(x, 1-s,\omega) \, dS_x.
\]

The equation (\ref{eq:farfield}) gives a kind of Friedlander limit (see [Fr73]) so
\[
 - \frac{1}{2 \pi} \int_{x \cdot \theta = 1} (\theta \cdot \nabla u)(x, 1-s,\omega) \, dS_x
\]
is a good candidate for being called the far field pattern and we define the far field pattern
\[
\alpha(\theta, \omega, s) :=  - \frac{1}{2 \pi} \int_{x \cdot \theta =1} (\theta \cdot \nabla u)(x,1-s, \omega) \, dS_x,
\qquad \theta, \omega \in S, ~ s \in \R,
\]
which is, up to a constant multiple, the Radon transform, on the
 plane $x \cdot \theta = 1$, of the directional derivative of $u(x, 1-s, \omega)$ in the direction $\theta$.
So our goal is the recovery of $q(\cdot)$ from the backscattering data
\[
\alpha( - \omega, \omega, s)
= -\frac{1}{2 \pi} \int_{x \cdot \omega =-1} (\omega \cdot \nabla u)(x, 1-s, \omega) \, dS_x, 
\qquad \omega \in S, ~ s \in \R.
\]

We make a remark about the regularity of $\alpha(\theta, \omega, s)$. As can be seen from the proof of 
Theorem \ref{thm:forward}, $u(x,t, \omega)$ is a function supported on $t \geq x \cdot \omega$ and is the restriction to the region
$t \geq x \cdot \omega$ of a smooth function on $\R^3 \times \R$. Let $a(x,t,\omega)$ be one such smooth extension of $u$; then
$u(x,t,\omega) = a(x,t,\omega) H(t- x \cdot \omega)$ and hence
\[
(\theta \cdot \nabla u)(x,t,\omega) =  (\theta \cdot \nabla a)(x,t,\omega) H(t - x \cdot \omega) - (\theta \cdot \omega) a(x,t,\omega) \delta( t- x \cdot \omega).
\]
An analysis of the integral defining $\alpha(\theta, \omega, s)$ shows that it is smooth in $s, \omega,\theta$ over the region
where $\theta \neq \omega$ and $\theta \neq - \omega$ and is zero for $s>2$. One can also show that when $\theta = \omega$, 
the peak scattering case, as a distribution on $s \in \R$, we have
\[
\alpha( \omega, \omega, s) =  \frac{1}{4 \pi}\left ( \int_{\R^3} q(x) \, dx \right ) \, \delta(s) + \text{smoother terms}
\]
and is zero for $s>0$. When $\theta = -\omega$ then $\alpha(-\omega, \omega,s)$ is a smooth function of $s$ and is zero for $s>2$.

We introduce some notation to state our main result. To any $\rho \geq 0$ and $\omega \in S$, we associate a unique 
$x=\rho \omega \in \R^3$. 
We define the 
angular derivatives
$\Omega_{ij} = x_i \pa_j - x_j \pa_i$ for $i,j=1,2,3$, $i \neq j$ and note that 
\[
\Delta_S := \sum_{i < j} \Omega_{ij}^2
\]
is the spherical Laplacian, that is $\Delta_S$ is the Laplace-Beltrami operator on the unit sphere. Further
\beqn
\int_S ( \Omega_{ij} f) \, g \, dS = - \int_S f \, ( \Omega_{ij} g) \, dS
\label{eq:ibyp}
\eeqn
for arbitrary smooth functions $f,g$ on $S$. 

Our main result is the following uniqueness theorem for the inverse backscattering problem.
\begin{theorem}[Uniqueness for back-scattering data]\label{thm:backscatter}
Suppose $q_i$, $i=1,2$ are smooth functions on $\R^3$ with support in the unit ball $B$ and
$\alpha_i( \cdot, \cdot, \cdot)$ the corresponding far field data. If there is a constant $C$, independent of $\rho$ and $i,j$ 
such that
\beqn
\int_S | \Omega_{i,j} (q_1 - q_2)(\rho \omega)|^2 \, d \omega
\leq C \int_S | (q_1- q_2)(\rho \omega) |^2 \, d \omega, \qquad \forall \rho \in [0,1], ~ \forall ~ i,j=1,2,3
\label{eq:angular}
\eeqn
then $\alpha_1(-\omega, \omega, s) = \alpha_2(-\omega, \omega, s)$ for all $ \omega \in S$ and all $s \in [0,2]$ implies $q_1=q_2$.
\end{theorem}

Let $\{ \phi_n(\omega) \}_{n \geq 1}$ denote an orthonormal basis for $L^2(S)$ consisting of spherical harmonics. Each $\phi_n(\omega)$ is the restriction to $S$ of a homogeneous harmonic polynomials $\phi_n(x)$, and the $\phi_n$ are indexed so that if $m<n$ then 
$\text{deg}( \phi_m) \leq \text{deg}( \phi_{n} )$; further 
\beqn
\Delta_S \phi_n = - d_n(d_n+1) \phi_n
\label{eq:DSphi}
\eeqn
where $d_n= \text{deg}(\phi_n)$ -  see [Se66] and [SW71] for details.
If $p(x)$ is a smooth function on $\R^3$ then $p$ has a spherical harmonic expansion
$
p(\rho \omega) = \sum_{n=1}^\infty p_n(\rho) \phi_n(\omega).
$
One can show\footnote
{
Since
$
(\Omega_{ij} p ) (\rho \omega) = \sum_{n=1}^\infty p_n(\rho) \, (\Omega_{ij} \phi_n ) (\omega)
$
so
\begin{align*}
\sum_{i<j} \int_S  &  (\Omega_{ij} p ) (\rho \omega)^2 \, d \omega
 = \sum_{m,n=1}^\infty p_n(\rho) \, p_m(\rho) \int_S \sum_{i<j} 
(\Omega_{ij} \phi_m ) (\omega) (\Omega_{ij} \phi_n ) (\omega) \, d \omega
\\
& = - \sum_{m,n=1}^\infty p_n(\rho) \, p_m(\rho) \int_S \sum_{i<j} 
(\Omega_{ij}^2 \phi_m ) (\omega)  \, \phi_n  (\omega) \, d \omega
 = - \sum_{m,n=1}^\infty p_n(\rho) \, p_m(\rho) \int_S 
(\Delta_S \phi_m ) (\omega)  \, \phi_n  (\omega) \, d \omega
\\
& = \sum_{m,n=1}^\infty p_n(\rho) \, p_m(\rho) \, d_m( d_m+1) \int_S 
 \phi_m  (\omega)  \, \phi_n  (\omega) \, d \omega
= \sum_{n=1}^\infty d_n ( d_n + 1) \, p_n(\rho)^2 .
\end{align*}
}
 that $p(x) = (q_1-q_2)(x)$ satisfies the angular derivative condition (\ref{eq:angular}) in Theorem \ref{thm:backscatter} iff we can find $C$ (independent of $\rho$) so that
\beqn
\sum_{n=1}^\infty d_n ( d_n + 1) \, p_n(\rho)^2 
\leq C \sum_{n=1}^\infty p_n(\rho)^2, \qquad \forall \rho \in [0,1].
\label{eq:angcond}
\eeqn
Clearly (\ref{eq:angcond}) holds if $p_n(\cdot) =0$ for all $n \geq N$ for some $N$, but (\ref{eq:angcond}) also holds for some $p$ 
with infinite spherical harmonic expansions. In fact, one may show that $d_n < \sqrt{n}$, so if we take $p_1(\rho)$ to be some non-
zero function, and choose $p_n(\rho)$ so that
\[
(\sqrt{n} + 1) |p_n(\rho)| \leq |p_{n-1}(\rho)|, \qquad \forall n \geq 2, ~ \rho \in [0,1]
\]
then $p$ would satisfy (\ref{eq:angcond}) for some $C$.

Theorem \ref{thm:backscatter}, shows, in particular, that two radial $q(x)$ are identical if their backscattering data are identical; even
this result, for the special case of radial potentials, is new. 
If $\alpha(\theta, \omega,s)$ is the far field pattern of $q(x)$ and $\beta(\theta, \omega, s)$ the far field pattern of its translate 
$q(x+a)$, $a \in \R^3$, then (see subsection \ref{subsec:translate})
\[
\beta(\theta, \omega, s) = \alpha(\theta, \omega,s + a \cdot(\theta-\omega)).
\]
This can be used to show that if the backscattering data for $q_1(|x|)$ equals the back-scattering data for $q_2(|x-a|)$ then both 
these functions are zero. In fact, if the far field patterns of $q_1(|x|)$ and $q_2(|x|)$ are $\alpha_1$ and $\alpha_2$ respectively, 
then the backscattering data for $q_1(|x|)$ and $q_2(|x-a|)$ are $\alpha_1(-\omega,\omega,s)$ and 
$\alpha_2(-\omega,\omega, s - 2 a \cdot \omega)$ respectively. So if
\[
\alpha_1(-\omega, \omega,s) = \alpha_2(-\omega, \omega, s - 2a\cdot \omega ),
\qquad \forall \omega \in S, ~ s \in \R,
\]
and noting that $\alpha_i(-\omega, \omega, s)$, $i=1,2$ are independent of $\omega$ because $q_i$ are radial, we obtain that
$\alpha_i(-\omega, \omega, s)$ are independent of $s$. This forces $\alpha_i(-\omega, \omega,s)=0$ for all $\omega,s$ because
far field patterns are always zero for $s$ large enough. Hence 
from Theorem \ref{thm:backscatter}, applied to the radial case, we obtain $q_1=q_2=0$.

The proof of Theorem \ref{thm:backscatter} relies on two ideas. We use an identity obtained by using the solution of an adjoint problem, an idea used earlier by Santosa and Symes in [SnSy88], and by Stefanov in [St90]. Also, for functions $f$ on $\R^3$, we estimate the $L^2$ norm of $f$ on spheres by the Radon transform of $f$ on planes outside the sphere using an idea motivated by the material on pages 185-190 in [LRS86]. The Radon transform estimate could also be obtained using Dean's theorem - see Chapter 7 in [Is06] - but we get stronger results using the idea in [LRS86].


Next we give some elementary, known but interesting, results with proofs which are perhaps a little simpler than the 
original proofs of these results.
\begin{theorem}(Elementary results)\label{thm:elementary}
Suppose $q_i$, $i=1,2$ are smooth functions on $\R^3$ with support in the unit ball and
 $\alpha_i(\cdot, \cdot, \cdot)$ the corresponding far field patterns. 
\vspace{-0.2in}
\begin{enumerate}
\item[(a)] If $q_2 \geq q_1$ and $\alpha_1(-\omega,\omega, s) = \alpha_2(-\omega,\omega, s)$ for a 
fixed $\omega\in S$ and all $s \in [-2,2]$,  then $q_1 = q_2$.
\item[(b)] There is an $M>0$ such that if $\|q_i\|_{C^2(\R^3)} \leq M$ for $i=1,2$ and
$\alpha_1(-\omega,\omega, s) = \alpha_2(-\omega,\omega, s)$ for all $\omega \in S$ and all $s \in [0,2]$ 
then $q_1 = q_2$.
\end{enumerate}
\end{theorem}

The result (a) was proved in [St90]. Melrose and Uhlmann in [Uh01], [MU08] and Lagergren in Chapter 8 of  [La01], [La11] have
shown results analogous to (b), for a different norm, though the result in [La01], [La11] is for the Schr$\oml$dinger equation. We think our 
proof is simpler.
Analogous to (b), the articles [SU97],  
[Wa98], [Wae98], [Wam98], [Wa00], study the inverse backscattering problem but for the acoustic equation, Maxwell's equation or the 
equations of elasticity and prove injectivity or stability for the problem when the coefficients are close to a constant.

{\em Sadly, even the most basic question remains open: if for some smooth, compactly supported $q$, the backscattering data $\alpha(-\omega, \omega, s)=0$ for all $\omega \in S$ and all $s \in \R$, then is $q=0$?}

\subsection{History}\label{subsec:history}

The term ``scattering data'' has been used in at least five other contexts and we summarize the connections between them. 
Our scattering data
$\alpha(\tau, \theta, \omega)$ is very close to the scattering kernel $k_q(s, \theta, \omega)$ defined in the 
Lax-Phillips scattering theory;  one can show that (see [Uh01])
\[
-2 \pi \, k_q(s, \theta, \omega) = \alpha_{\tau \tau}(\theta, \omega, s) \qquad \forall s \in \R, ~ \theta \in S, ~
 \omega \in S.
\]

For each real number $k>0$ and unit vector $\omega \in S$, let $w(x,\omega,k)$ be the outgoing
 solution of the Helmholtz equation corresponding to the incoming wave 
$w_i(x,\omega,k) = e^{i k x\cdot \omega}$, that is $w$ is the solution of (below $\rho=|x|$)
\begin{gather}
(-\Delta_x + q(x) - k^2)w(x,\omega, k) = 0, \qquad x \in R^3
\label{eq:wpde}
\\
\lim_{\rho \rightarrow \infty} \rho \left ( \frac{\partial w_s}{\partial \rho}
- i k w_s \right)(x, \theta, k) ~=~0,
\label{eq:rad}
\end{gather}
where $w_s$ is the scattered part of the solution
\[
w_s(x,\omega,k) := (w - w_i)(x,\omega,k).
\]
For large $|x|$ 
\beqn
w_s(x,\omega, k) = \frac{e^{ik |x|}}{|x|}
w_\infty \left(\frac{x}{|x|}, \omega, k \right ) +
o\left (\frac{1}{|x|} \right)
\label{eq:ffp}
\eeqn
and the function $w_\infty(\theta, \omega, k)$, $\omega, \theta \in S$ and  $k>0$, is
called the far field pattern associated to $q(.)$ - see [CK13] for details.  One may show that $w_\infty(\theta, \omega, k)$
is a constant multiple of the Fourier transform of our time domain scattering data $\alpha(\theta,\omega,\tau)$ - see 
[Uh01]. 

The far field pattern $w_\infty(\theta, \omega,k)$ for unit vectors $\theta, \omega$ and $k>0$ determines a function
$h(\xi,\eta,k)$ for all $\xi, \eta \in \R^3$ with $|\xi|=|\eta| =k \neq 0$  via
\[
w_\infty( \theta, \omega, k) = h(k \theta, k \omega, k).
\]
In [ER92], by regarding $h$ as the solution of a certain integral equation, they extend $h(\xi, \eta, k)$ to a function on
$\R^3 \times \R^3 \times [0, \infty)$ - they even allow complex potentials $q$. They choose the function $h(\xi, -\xi, |\xi|))$ 
with $\xi \in \R^3$ as their backscattering data; loosely speaking, this corresponds to an extension of $w_\infty(- \omega, 
\omega, k)$ to the set $k \geq 0$ for all unit vectors $\theta, \omega$ - so $k=0$ is now included in the domain.

Another candidate for the scattering data is the standard scattering operator $S_q$ defined via wave operators arising from 
the solution operator of the initial value problem for the wave equation with zeroth order coefficient $q$ - see [Uh01].
If $T_q$ is the operator with kernel $k_q(s-s', \theta, \omega)$, then in [Uh01] it is shown that $S_q$ is the conjugate of 
$I+T_q$ by the  modified Lax-Phillips Radon transform.

In [BM09], [La01], [La11], the scattering operator is generated exactly as in the definition of $S_q$ above but with the use of 
the solution operator of the initial value problem for the Schr$\oml$dinger equation rather than the wave equation. In [La01], 
[La11] they do relate their scattering data to solutions of the wave equation and we have tried to use this connection to establish a 
relationship between their scattering data and our scattering data. We have come close to doing so but we have not succeeded fully so 
we have not included this part of our work in this article.

The inverse backscattering problem in the various contexts mentioned above consists of inverting the map sending $q$ to one of the 
above forms of the backscattering data. In [ER92] it was shown that the map $q(x) \to h(\xi, -\xi, |\xi|)$ is an analytic map in 
appropriate spaces and this map is an isomorphism on a dense open subset of this space (which includes a neighborhood of $q=0$). 
The articles [St92], [Uh01] (the details of [Uh01] are given in [MU08]) give analogous results for the map
$q(x) \to w_\infty(-\omega, \omega, k)$, [La01] has related results and [Wa02] has similar results for the even dimensional case. 
[Uh01], [MU08] go a little farther; there an explicit series expansion (the Born series expansion) is given for the backscattering 
map and the series is shown to be convergent for compactly supported potentials in $H^2(\R^3)$.

One may also study what can be recovered of $q$ from the backscattering data for a single frequency. Many different $q$ can result in 
the same backscattering data for a fixed frequency, but in [HKS05] it was shown that a certain subset (possibly empty) of $\R^3$, 
determined by the data, is guaranteed to be in the convex hull of the support of all such $q$.

A formal computation of the derivative of the map $q \to \alpha(-\omega, \omega, s)$, at $q=0$, shows (see subsection
\ref{subsec:born}) that this map is 
\[
p(x) \to  \frac{-1}{8 \pi} \int_{ x \cdot \omega = -2s} p(x) \, dS_x,
\]
which may be interpreted as saying that, for small $q$, we have
\[
\alpha(-\omega,\omega,s)    \approx    \frac{-1}{8 \pi} \int_{ x \cdot \omega = -2s} p(x) \, dS_x.
\]
Hence, using the inverse of the Radon transform, a candidate to approximate $q(x)$ (constructed from backscattering data) 
would be
\[
q_b(x)  := \frac{1}{4\pi} \int_S \alpha_{ss}(-\omega, \omega, - x \cdot \omega/2) \, d \omega,
\]
which is called the Born approximation to $q$. 
This is identical and analogous to what is done in the frequency domain case, where the derivative of the map
$q \to w_\infty( - \omega, \omega, k)$, at $q=0$, is the map 
\[
p(x) \to \hat{p}(2k \omega),
\]
so the Born approximation is defined as
\[
q_b(x) := \int_0^\infty \int_S  k^2 e^{-ikx \cdot \omega}  w_\infty( -\omega, \omega, k/2) \, d \omega \, dk.
\]

If a plane wave impinges on a potential $q$ which has a singularity across a surface, the transmitted wave is the same as the
 original but the 
reflected wave is one degree smoother. So if a medium is probed by a plane wave then the resulting wave is a sum of the original 
wave plus a sum of waves which are the result of one or more reflections. Amongst the reflected waves, the waves resulting from
 a single reflection will be the most 
singular, those resulting from two reflections will be one degree smoother, those resulting from three reflections will be two degrees 
smoother and so 
on. Since $q_b$ is the result of applying the inverse of the single reflection process to the backscattering data, one expects $q_b$ and 
$q$ to the have the same principal singularity. This idea was implemented in [GU93] to show that if $q$ is a conormal potential then
one can recover the conormal singularities of $q$ from the singularities of the backscattering amplitude. Using tools from Harmonic 
Analysis, in [OPS01] for two dimensions and then in [RV05] for three dimensions, it was shown that for {\em arbitrary} (not 
necessarily conormal) smooth enough 
$q$, $q_b-q$ is smoother than $q$, that is $q_b$ captures the principle singularities of $q$. Please see [RR12] for an accurate 
statement of the most recent results - also see [BM09] for related results. Along these lines, [DUV] has a result for the (harder to 
analyze) acoustic equation. There it was shown that the reflected wave is smoother than the transmitted wave for conormal sound 
speeds in 
$C^{1+\ep}(\R^2)$ with $\ep>0$.

In [Ramm10] it is claimed that the map $q \to w_\infty(-\omega, \omega, k)$ is injective when $q$ is restricted to compactly 
supported real valued $q$. However, there is a gap in the proof. The complex geometrical optics (CGO) solutions are defined for 
complex $k$ but they are not analytic in $k$. The classical scattering solutions can be defined for complex $k$ and are analytic in $k$ 
but they are different from the CGO solutions. In [Ramm10] it is assumed that they are the same. 

Below $\cleq$ denotes `less than or equal to a constant multiple' with the constant independent of the parameters.


\section{Proof of Theorem \ref{thm:forward}}\label{sec:forward}
\begin{enumerate}
\item[(a)]
We seek $U(x,t,\omega)$ in the form
\[
U(x,t,\omega) = \delta(t-x \cdot \omega) + u(x,t,\omega) H(t-x \cdot \omega);
\]
then
\[
U_{tt} - \Delta U + q U = ( u_{tt} - \Delta u + qu) H(t - x \cdot \omega)
+ 2 (u_t + \omega \cdot \nabla u + q/2) \delta(t-x \cdot \omega).
\]
So we need to choose a smooth $u(x,t)$ on the region $ t \geq x \cdot \omega$ so that $u_{tt} - \Delta u + qu =0$ on this region and $u_t + \omega \cdot \nabla u =-q/2$ on the plane $x \cdot \omega = t$. The last relation is equivalent to
\[
\frac{d}{d\sigma} u (x + \sigma \omega, x \cdot \omega + \sigma, \omega) 
= -\frac{ q(x + \sigma \omega)}{2}
\]
and integrating it with respect to $\sigma$ and noting that $u(x,t, \omega) =0$ for $t<-1$, we obtain
\[
u(x, x \cdot \omega, \omega) = - \frac{1}{2} \int_{-\infty}^0 q(x + \sigma \omega) \, d \sigma.
\]
So we need to show that the characteristic IVP (\ref{eq:ude})-(\ref{eq:uic}) has a unique smooth solution. The uniqueness of the solution may be proved by standard energy estimates. The existence is proved by using a progressing wave expansion and converting the problem to the solution of an initial value problem.

We give an outline of the proof - the details can be filled in quite easily. Below, for any $j \geq 0$,
\[
s_+^j = \begin{cases} s^j, & s \geq 0 \\ 0, & s<0. \end{cases};
\]
note that $s_+^0 = H(s)$. Pick any positive integer $N$; we seek a solution $U(x,t)$ of (\ref{eq:Ude}), (\ref{eq:Uic}) in the form of a progressing wave expansion
\[
U(x,t) = \delta(t -x \cdot \omega) + \sum_{j=0}^N a_j(x) (t- x \cdot \omega)_+^j + R_N(x,t).
\]
The $a_j(x)$ are constructed by solving the associated transport equations - see [CH89].
One may show that the $a_j(x)$ are smooth functions completely determined by $q$ and its derivatives of order $j+2$ or less and $R_N(x,t)$ is the solution of the IVP
\begin{gather}
(\Box - q ) R_N(x,t)  = F(x) (t - x \cdot \omega)_+^N, \qquad (x,t) \in \R^3 \times ]R,
\label{eq:RNde}
\\
R_N(x,t) =0, \qquad x \in \R^3, ~ t < -1,
\label{eq:RNic}
\end{gather}
for some smooth function $F(x)$ completely determined by $q$ and its derivatives up to order $2N$. Since the RHS of
(\ref{eq:RNde}) is of class $C^{N-1}$ on $\R^3 \times \R$, by the well-posedness theory for hyperbolic PDEs (using integral equation arguments), the system (\ref{eq:RNde}), (\ref{eq:RNic}) has a unique solution of class $C^{N-1}$. Next one may check that
\[
u(x,t) =  \sum_{j=0}^N a_j(x) (t- x \cdot \omega)^j + R_N(x,t), \qquad t \geq x \cdot \omega
\]
solves (\ref{eq:ude}) - (\ref{eq:uic}), so we have proved the existence of a $u$ of class $C^{N-1}$ for every $N$. The uniqueness of $u$ allows us to claim that $u$ is smooth on $t \geq x \cdot \omega$.

\item[(c)] We prove (c) before (b) because the proof of (b) is more complicated.
We shorten $u(x,t,\omega)$ to $u(x,t)$, assume that $\theta = (0,0,1)$, write $x = (x',z)$ with $x' \in \R^2$, $z \in \R$, and define  
\[
v(z,t) := \int_{\R^2} u(x',z,t) \, dx'.
\]
Since $u_{tt} - \Delta u = 0$ in the region $|x| \geq 1$ and $u(x,t) = 0$ for $t \leq -1$, in the region
$z \geq 1$ we have
\[
v_{tt} - v_{zz} = \int_{\R^2} (u_{tt} - u_{zz})(x',z,t) \, dx' = - \int_{\R^2} (\Delta_{x'} u)(x',z,t) \, dx'
=0,
\]
and $v(z,t)=0$ for $t \leq -1$. Hence $v(z,t) = f(t-z)$, on the region $z \geq 1$, for some function $f$. Hence, for $z \geq 1$, we have
\begin{align*}
\int_{\R^2} u(x',z,t) \, dx' = v(z,t) = f(t-z) = v(1, t-z+1) = \int_{\R^2} u(x', 1, t-z+1) \, dx'
\end{align*}
proving one part of (c).
Next
\[
\int_{x \cdot \theta=1} (\theta \cdot \nabla u)(x,t) \, dS_x = 
\int_{\R^2} u_z(x',1,t) \, dx' = v_z(1,t) = - f'(t-1);
\]
and
\begin{align*}
\pa_t  \left ( \int_{x \cdot e=1} u(x,t) \, dS_x \right ) & = 
\pa_t \left ( \int_{\R^2} u(x',1,t) \, dx' \right )= \pa_t ( v(1,t) ) = \pa_t ( f(t-1) )
\\
& = f'(t-1) =  - \int_{-\infty}^t v_z(1,\tau) \, d \tau,
\end{align*}
proving the other part of (c).

\item[(b)]
As before, we shorten $u(x,t,\omega)$ to $u(x,t)$, assume that $\theta = (0,0,1)$, and write $x = (x',z)$ with 
$x' \in \R^2$, $z \in \R$.
Let $f(x',t) = u_z(x',z=1,t)$, that is $f$ is the 
value of $u_z$ on the hyperplane $x \cdot \theta = 1$. We note two properties of the distribution $f(x',t)$; we have
 $f(x',t)=0$ for all $t \leq -1$ and
for each $T$, the intersection of the support of $f$ with the region $t \leq T$ is compact.

Now $u$ is the solution of the IBVP(Initial Boundary Value Problem)
\begin{align}
u_{tt} - \Delta u = 0, &\qquad (x,t) \in \R^3 \times \R, ~~ z \geq 1,
\label{eq:ibde}
\\
u(x,t) = 0, &\qquad t <-1,
\label{eq:ibic}
\\
u_z(x',1,t) = f(x',t), &\qquad (x',t) \in \R^2 \times \R.
\label{eq:ibbc}
\end{align}
We may show that (see [Rak03] for example)
\begin{align*}
u(x',z+1,t) &=  - \frac{1}{2 \pi} \frac{ \delta( t - |(x',z)|)}{|(x',z)|} * f(x',t)
= - \frac{1}{2 \pi} \int_{R^2 \times \R} f(y',\tau) \frac{ \delta( t-\tau - | (x'-y',z)|) }{ |(x'-y',z)|} \, dy' \, d \tau
\\
&= - \frac{1}{2 \pi} \int_{\R^2} \frac{ f(x'+y', t - |(y',z)|)}{|(y',z)|} \, dy'
\end{align*}
as  a distribution in $(x',z,t)$ on the region $\R^2 \times (0, \infty) \times \R$. Hence using the pullback of the map
$(z,\sigma) \to (0,z,z+\sigma)$ we have
\[
z \, u(0,z+1,z+\sigma) = - \frac{z}{2 \pi}  \int_{\R^2} \frac{ f(y', z+\sigma - |(y',z)|)}{|(y',z)|} \, dy'.
\]
as a distribution in $(z, \sigma)$ on the region $(0,\infty) \times \R$. This distribution may also be regarded as a continuous map from
$(0, \infty)$ to $\D'(\R)$ (the distributions on $\R$) sending
\[
z \to - \frac{z}{2 \pi}  \int_{\R^2} \frac{ f(y', z+\sigma - |(y',z)|)}{|(y',z)|} \, dy'
\]
because for any compactly supported smooth function $\phi(\sigma)$ on $\R$
\begin{align*}
- \frac{z}{2 \pi} \int_\R \int_{\R^2} \frac{ f(y', z+\sigma - |(y',z)|)}{|(y',z)|}  \, \phi(\sigma) \, dy' \, d \sigma
 =
 - \frac{z}{2 \pi} \int_\R \int_{\R^2}  f(y', \sigma) \, \frac{ \phi( \sigma + |(y',z)| - z )}{ |(y',z)|}   \, dy' \, d \sigma
\end{align*}
is a 
continuous\footnote{The``integration'' in the last integral is over a compact region region in
$y',\sigma$ space determined by the support of $\phi$ and that $f(\cdot,t)=0$ for $t<-1$. So the map $z \to
\frac{ \phi( \sigma + |(y',z)| - z )}{ |(y',z)|}$ is a continuous map from $(0,\infty)$ to the space of test functions in the $y',\sigma$ variables.}
function of $z$ on $(0,\infty)$.

Now we show that, as distributions in $\sigma \in \R$,
\beqn
\lim_{z \rightarrow \infty} 
z \int_{\R^2} \frac{ f(y', z+\sigma - |(y',z)|)}{|(y',z)|} \, dy'
 =  \int_{\R^2} f(y',\sigma) \, dy'
\label{eq:ff}
\eeqn
which will imply that
\[
\lim_{z \rightarrow \infty}  z u(0,z+1,z+\sigma) =  - \frac{1}{2 \pi} \int_{\R^2} f(y',\sigma) \, dy',
\qquad
 \lim_{z \rightarrow \infty}  u(0,z+1,z+\sigma)  = 0
\]
and hence taking $\sigma=1-s$ we have
\[
\lim_{z \rightarrow \infty} (z+1) \, u(0,z+1,z+1-s) =  - \frac{1}{2 \pi} \int_{\R^2} f(y',1-s) \, dy',
\]
proving (b). So it remains to prove (\ref{eq:ff}).

For any compactly supported smooth function $\phi(\sigma)$ on $\R$, we have 
\beqn
z \int_\R \int_{\R^2} \frac{ f(y', z+\sigma - |(y',z)|)}{|(y',z)|} \, \phi(\sigma) \, dy' \, d \sigma
=
\int_\R \int_{\R^2} f(y',t) \,  \phi( t + |(y',z)| - z) \, \frac{z}{|(y',z)|} \, dy' \, dt.
\label{eq:temp3}
\eeqn
If $\phi$ is supported on $|\sigma| \leq R$ then the ``integration'' occurs on a subset of the region $ t + |(y',z) - z \leq R$ and hence
on a subset of $t \leq R$. Since the intersection of the support of $f$ with the region $t \leq R$ is compact, the ``integration'' in
(\ref{eq:temp3}) occurs on a compact subset of $\R \times \R^2$.  Next, noting that 
\[
| (y',z)| - z = \frac{|y'|^2}{ |(y',z)| + z} \leq \frac{|y'|^2}{z}
\]
one can show that as $z \to \infty$ we have $ z/|(y',z)| \to 1$ in the $C^k$ norm on compact subsets of $\R^2$,
 for all $k \geq 0$. 
Further,
from the mean value theorem, for any smooth, compactly supported function $\psi(t)$ on $\R$ we have
\[
|\psi( t + |(y',z)| - z) - \psi(t)| \leq M ( |y',z| - z)
\]
so that as $z \to \infty$ we have $ \phi( t + |(y',z)| - z) \to \phi(t)$ in the $C^k$ norm on compact subsets of $\R^2 \times \R$, for all $k \geq 0$. Hence (\ref{eq:ff}) follows from the continuity property of distributions.
\end{enumerate}


\section{A useful identity}

We derive an identity used in the proofs of Theorems 2 and 3.

Let $U_i$, $i=1,2$, be the solution of (\ref{eq:Ude}), (\ref{eq:Uic}) when $q=q_i$, and 
let $\alpha_i$ be the far field pattern associated to $q_i$. Define $v := U_1 - U_2
= u_1 - u_2$, $p := q_2 - q_1$ and $\alpha:= \alpha_1 - \alpha_2$; then
\begin{align}
v_{tt} - \Delta u + q_1 v = p U_2, & \qquad (x,t) \in \R^3,
\label{eq:vde}
\\
v(x,t) = 0, & \qquad x \in \R^3, ~t \leq -1.
\label{eq:vic}
\end{align}
We show that $v$ and $\alpha$ satisfy the following identity.
\begin{proposition} \label{prop:identity}
For any $\tau \in \R$ and all $\omega \in S$ we have
\beqn
8 \pi  \alpha(-\omega, \omega,- 2 \tau ) = \int_{x \cdot \omega = \tau} p(x) \, dS_x
+ \int_{-1}^\tau \int_{x \cdot \omega = t}  k(x, \omega, \tau) \, p(x) \, dS_x \, dt,
\label{eq:q1q2iden}
\eeqn
 where
\[
k(x, \omega, \tau) := 2 (u_1 + u_2)(x, 2 \tau- x \cdot \omega, \omega) + 
2 \int_{x \cdot \omega}^{2 \tau - x \cdot \omega}  u_1(x,s, \omega) \,
u_2(x, 2 \tau -s, \omega) \, ds,
\]
 is smooth on the region $ -1 \leq x \cdot \omega \leq \tau$ with $\tau \in \R$, $\omega \in S$ and $x \in \R^3$.
\end{proposition}
\begin{proof}
Since the $q_i$ are supported in the unit ball, both sides of (\ref{eq:q1q2iden}) are zero if $\tau < -1$, so we focus on the $\tau \geq -1$ case. Choose any $\tau \geq -1 $ and define $W_1(x,t) = U_1(x, 2 \tau - t, \omega)$; then $W_1$ satisfies (\ref{eq:Ude}) with $q$ replaced by $q_1$ and noting that $U_1(x,t, \omega) = \delta(t- x \cdot \omega)$ for $ t \leq -1$ we have
\[
W_1(x,t) = \delta( 2 \tau  -t - x \cdot \omega) \qquad \text{for} ~ t \geq 2 \tau + 1.
\]
Noting that for each $t \in [-1, 2 \tau+1]$, $v(x,t)$ is compactly supported as a function of $x$,
working formally (which can be made rigorous by integrating 
$p(x) \, u_2(x,t, \omega) \, u_1(x, 2 \tau-t, \omega)$ over the region 
$x \cdot \omega \leq t \leq 2 \tau - x \cdot \omega$) we have
\begin{align}
\int_{ \R^3} \int_{-1}^{2 \tau + 1} & p(x) \, U_2(x,t, \omega) \, W_1(x,t) \, dt \, dx
= \int_{\R^3} \int_{-1}^{2 \tau + 1} (v_{tt} - \Delta v + q_1 v)(x,t) \, W_1(x,t) \, dt \, dx
\nn
\\
& = \int_{\R^3} (v_t W_1 - W_{1t}v)(x, 2 \tau+1) \, dx
-  \int_{\R^3} (v_t W_1 - W_{1t} v)(x, -1) \, dx 
\nn
\\
& = \int_{\R^3} v_t(x, 2 \tau+1) \, \delta( -1 - x \cdot \omega ) \, dx
+ \int_{\R^3} v(x, 2 \tau+1) \, \delta'(-1 - x \cdot \omega) \, dx
\nn
\\
& = \int_{x \cdot \omega = -1} v_t(x, 2 \tau + 1) \, dS_x
- \int_{\R^3} v(x, 2 \tau + 1) \,  ( \omega \cdot \nabla) ( \delta( -1 - x \cdot \omega ) )  \, dx
\nn
\\
& = \int_{x \cdot \omega = -1} v_t(x, 2 \tau + 1) \, dS_x
+ 
\int_{x \cdot \omega = -1} ( \omega \cdot \nabla v)(x, 2 \tau + 1) \, dS_x
\nn
\\
& = 4 \pi \alpha(-\omega, \omega, -2 \tau)
\label{eq:ttuse}
\end{align}
with the last step following from (\ref{eq:DN}) of Theorem \ref{thm:forward} with $\theta= - \omega$ and the definition of $\alpha$.
We now analyze the LHS of (\ref{eq:ttuse}). For $\tau \geq -1$, we have
\begin{align*}
\int_{ \R^3} \int_{-1}^{2 \tau + 1}  &  p(x) \, U_2(x,t, \omega) \, W_1(x,t) \, dt \, dx
= \int_{ \R^3} \int_{-1}^{2 \tau + 1} p(x) \, U_2(x, t, \omega) \, U_1(x, 2 \tau -t, \omega) \, dt \, dx
\\
& =  \int_{ \R^3} \int_{-1}^{2 \tau + 1} p(x) \, \delta( t - x \cdot \omega) \, \delta( 2 \tau -t - x \cdot \omega) \, dt \, dx
\\
& \qquad+ \int_{ \R^3} \int_{-1}^{2 \tau + 1} p(x) \, \delta( t - x \cdot \omega) \, u_1(x, 2 \tau -t, \omega)
 \, dt \, dx
\\
& \qquad+ \int_{ \R^3} \int_{-1}^{2 \tau + 1} p(x) \, u_2(x,t, \omega) \, \delta( 2 \tau -t - x \cdot \omega) \, dt \, dx
\\
& \qquad+
\int_{ \R^3} \int_{-1}^{2 \tau + 1} p(x) \, u_1(x,2 \tau-t, \omega) \, u_2(x, t, \omega)  \, dt \, dx.
\end{align*}
The first of these four integrals on the RHS is the Radon transform of $p$. In the second integral, 
using the support of $u_1$, the region of integration is
$ x \cdot \omega \leq 2 \tau -t$ where $x \cdot \omega = t$ and hence $t \leq \tau$. In the third integral
 the region of integration is $ x \cdot \omega \leq t$ where $ x \cdot \omega
= 2 \tau -t$ and hence $ \tau \leq t$. In the fourth integral the region of integration is
$ x \cdot \omega \leq t$ and $x \cdot \omega \leq 2 \tau -t$ so adding these two we get $x \cdot \omega \leq \tau$. Hence
\begin{align}
\int_{ \R^3} \int_{-1}^{2 \tau + 1}  &  p(x) \, U_2(x,t, \omega) \, W_1(x,t) \, dt \, dx
= \frac{1}{2} \int_{ x \cdot \omega = \tau} p(x) \, dS_x
+  \int_{-1}^\tau  \int_{x \cdot \omega = t} p(x) u_1(x, 2 \tau -t, \omega) \, dS_x \, dt
\nn
\\
&  \qquad + \int_{\tau}^{2\tau+1} \int_{ x \cdot \omega = 2 \tau -t} p(x)  \, u_2(x, t, \omega) \, dS_x \, dt
\nn
\\
& \qquad  + \int_{-1 \leq x \cdot \omega \leq \tau} \int_{x \cdot \omega}^{2 \tau - x \cdot \omega}
p(x) \, u_1(x,t, \omega) \, u_2(x, 2 \tau -t, \omega)  \, dt \, dx
\nn
\\
& = \frac{1}{2} \int_{ x \cdot \omega = \tau} p(x) \, dS_x
+
 \int_{-1}^\tau  \int_{x \cdot \omega = t} p(x) \, (u_1+u_2)(x, 2 \tau -t, \omega) \, dS_x \, dt
\nn
\\
& \qquad  + \int_{-1 \leq x \cdot \omega \leq \tau} \int_{x \cdot \omega}^{2 \tau - x \cdot \omega}
p(x) \, u_1(x,s, \omega) \, u_2(x, 2 \tau -s, \omega)  \, ds \, dx
\nn
\\
& =   \frac{1}{2} \int_{ x \cdot \omega = \tau} p(x) \, dS_x
+
 \int_{-1}^\tau  \int_{x \cdot \omega = t} p(x) \, (u_1+u_2)(x, 2 \tau -t, \omega) \, dS_x \, dt
\nn
\\
& \qquad + \int_{-1}^\tau \int_{x \cdot \omega = t} \int_{x \cdot \omega} ^{2 \tau - x \cdot \omega} p(x) \, u_1(x,s, \omega) \,
u_2(x, 2 \tau -s, \omega) \, ds \, dS_x \, dt
\nn
\\
& = \frac{1}{2} \int_{ x \cdot \omega = \tau} p(x) \, dS_x 
+ \frac{1}{2} \int_{-1}^\tau \int_{x \cdot \omega = t} p(x) \, k(x, \omega, \tau) \, dS_x \, dt
\label{eq:tempiden}
\end{align}
for all unit vectors $\omega$ and all $\tau \geq -1$; here
\[
k(x, \omega, \tau) := 2 (u_1 + u_2)(x, 2 \tau- x \cdot \omega, \omega) + 
2 \int_{x \cdot \omega}^{2 \tau - x \cdot \omega}  u_1(x,s, \omega) \,
u_2(x, 2 \tau -s, \omega) \, ds 
\]
in the region $\omega \in S$, $-1 \leq  \tau$ and $ x \cdot \omega \leq \tau$. Note that, in this region, $k(x,\omega, \tau)$ depends on the values of values of $u_1( \cdot, \cdot, \omega)$, and
$u_2(\cdot, \cdot, \omega)$ at points $(x',t')$ where
$t' \geq x' \cdot \omega $ because, in this region, for the first term $2 \tau - x \cdot \omega \geq x \cdot \omega$ and in the integral, $s \geq x \cdot \omega$ and $ 2 \tau- s \geq x \cdot \omega $. Hence $k(x,\omega, \tau)$ is a smooth function on this region. Combining (\ref{eq:ttuse} ) and (\ref{eq:tempiden}) we obtain
\[
8 \pi  \alpha(-\omega, \omega,- 2 \tau )  = \int_{ x \cdot \omega = \tau} p(x) \, dS_x 
+ \int_{-1}^\tau \int_{x \cdot \omega = t} p(x) \, k(x, \omega, \tau) \, dS_x \, dt
\]
which proves the proposition.
\end{proof}


\section{Proof of Theorem \ref{thm:backscatter}}



\subsection{ An expansion}
For $x \in \R^n$, define the vectors
\[
T_{ij} = x_i e_j - x_j e_i, \qquad  ~ i,j=1, \cdots,n,
\]
which are tangential, at $x$, to the origin centered sphere through $x$; here $e_i$ is the unit vector
along the $x_i$ axis. Note that 
\[
\Omega_{ij} := x_i \pa_j - x_j \pa_i  = T_{ij} \cdot \nabla.
\]
For any vector $v$ in $\R^n$, we express $v$ in terms of $x$ and the $T_{ij}$.
\begin{proposition}\label{prop:vperp}
For any $x,v \in \R^n$, we have
\beqn
|x|^2 v = \sum_{i < j} (v \cdot T_{ij}) T_{ij} + (v \cdot x) x.
\label{eq:vxT}
\eeqn
\end{proposition}
\begin{proof}
Let $v = ( v_1, \cdots, v_n)$ and $x = (x_1, \cdots, x_n)$; then taking the dot product of the RHS (\ref{eq:vxT}) with $e_k$ we obtain
\begin{align*}
e_k \cdot ( \text{RHS of (\ref{eq:vxT}}) )  & =
\sum_{i<j}  (v \cdot T_{ij}) (T_{ij} \cdot e_k ) + ( v \cdot x) x_k 
\\
& = \sum_{i<k} (v \cdot T_{ik}) (T_{ik} \cdot e_k ) +  \sum_{k<j} (v \cdot T_{kj}) (T_{kj} \cdot e_k ) +  ( v \cdot x) x_k 
\\
&=  \sum_{i<k} (v \cdot T_{ik}) x_i -  \sum_{k<j} (v \cdot T_{kj}) x_j +  ( v \cdot x) x_k 
\\
& =  \sum_{i<k} (v \cdot T_{ik}) x_i +  \sum_{k<j} (v \cdot T_{jk}) x_j +  ( v \cdot x) x_k 
\\
& = \sum_{i} ( v \cdot T_{ik} ) x_i +  ( v \cdot x) x_k 
\\
& = \sum_i (v_k x_i - v_i x_k) x_i + (v \cdot x) x_k
\\
& = v_k |x|^2.
\end{align*}
\qquad
\end{proof}

\subsection{ A derivative of the Radon transform}

For each $\tau \in \R$ and $\omega \in S$ and any smooth function $p(x)$ on $\R^3$ supported in $B$, we define the Radon transform
\[
P(\tau, \omega) := \int_{x \cdot \omega = \tau} p(x) \, dS_x.
\]
Hence, by the Divergence theorem
\[
P(\tau, \omega) = \int_{ x \cdot \omega \leq \tau} (\omega \cdot \nabla p)(x) \, dx
= \int_{-\infty}^\tau \int_{x \cdot \omega =t} (\omega \cdot \nabla p)(x) \, dS_x \, dt
\]
so (the $\tau$ partial derivative of $P$)
\beqn
P_\tau(\tau, \omega) = \int_{x \cdot \omega =\tau} (\omega \cdot \nabla p)(x) \, dS_x.
\label{eq:Pt}
\eeqn

Given $\omega \in S$ and $\tau \in [0,1]$, to any point $x \in \R^3$ on the plane $x \cdot \omega = \tau$ we associate $\rho=|x|$, and $(r,\theta)$ the polar coordinates of $x$ as points on the plane
$x \cdot \omega = \tau$ - so $r$ is the distance of $x$ from the line through the origin in the direction $\omega$; see Figure \ref{fig:decomp}. On the line through the origin and $\omega$ we choose a point $Q$ so that the vector $xQ$ is orthogonal to the vector $x$. Let $\alpha$ denote the unit vector in the direction $xQ$; our goal is to express, at $x$, the vertical directional derivative
$\omega \cdot \nabla p$ in terms of the radial derivative $p_r$ and the (angular) derivative in the
direction $\alpha$. 
\begin{figure}
\begin{center}
\epsfig{file=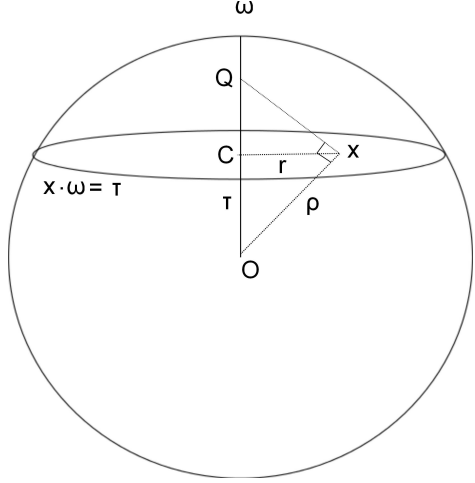, height=2.5in}
\end{center}
\caption{The decomposition of $\omega$}
\label{fig:decomp}
\end{figure}
From the similar triangles $OCx$ and $xCQ$ we have
\[
\frac{|xQ|}{|x|} =\frac{|CQ|}{|Cx|} = \frac{|Cx|}{|OC|}
\]
that is $|xQ| = \rho r/\tau$ and $|CQ|=r^2/\tau$. Now, as vectors we have $xQ=xC+CQ$ so
\[
\frac{\rho r}{\tau} \alpha = - r \hat{r} + \frac{r^2}{\tau} \omega
\]
where $\hat{r}$ is the unit vector in the radial direction at $x$, that is in the direction $Cx$. Hence
\[
\omega = \frac{\rho}{r} \alpha + \frac{\tau}{r} \hat{r}
\]
implying
\[
(\omega \cdot \nabla p)(x) = \frac{\tau}{r} p_r(x) + \frac{\rho}{r} (\alpha \cdot \nabla p)(x).
\]
Substituting this in (\ref{eq:Pt}) we obtain
\begin{align}
P_\tau(\tau, \omega) &=  \tau \int_0^{2 \pi} \int_0^\infty p_r \, dr \, d \theta
+ \int_{x \cdot \omega=\tau}  \frac{\rho}{r} (\alpha \cdot \nabla p)(x) \, dS_x
\nn \\
& = - 2 \pi \tau p(\tau \omega) +  
\int_{x \cdot \omega=\tau}  \frac{\rho}{r} (\alpha \cdot \nabla p)(x) \, dS_x.
\label{eq:ptt}
\end{align}
For a fixed $\tau \in [0,1]$, the plane $x \cdot \omega = \tau$ may be parametrized by $\rho$ and $\theta$ and we note that
\[
dS = r \, dr \, d \theta = \rho \, d \rho \, d \theta
\]
because $\rho^2 = r^2 + \tau^2$ gives $r \, dr = \rho \, d \rho$. So using the support of $p$ and (\ref{eq:ptt}) we obtain
\begin{align}
P_{\tau}(\tau, \omega) 
&= - 2 \pi \tau p(\tau \omega) + \int_\tau^1  \int_0^{2 \pi} 
\frac{\rho^2}{ \sqrt{\rho^2 - \tau^2}} \, (\alpha \cdot \nabla p)(x) \, d \theta \, d \rho.
\label{eq:Ptau}
\end{align}

Applying Proposition \ref{prop:vperp} to $v=\alpha$ and noting that $\alpha \perp x$ we have
\begin{align*}
|x|^2 ( \alpha \cdot \nabla p)(x) &= \sum_{i<j} (\alpha \cdot T_{ij}) ( T_{ij} \cdot \nabla p)(x)
=  \sum_{i<j} (\alpha \cdot T_{ij}) (\Omega_{ij} p)(x).
\end{align*}
Now $ |\alpha| =1$ and $|T_{ij}| \leq 2 |x|$, so $ |\alpha \cdot T_{ij}| \leq 2 |x|$ and hence
\[
|x|^2 \, | (\alpha \cdot \nabla p)(x)|
\leq 2|x| \sum_{i<j}  | (\Omega_{ij} p)(x) |.
\]
So (\ref{eq:Ptau}) leads to 
\[
\tau |p(\tau \omega)| \cleq  |P_\tau(\tau, \omega)| + \sum_{i<j} \int_\tau^1 \int_0^{2 \pi} 
\frac{\rho}{\sqrt{\rho^2 - \tau^2}} | (\Omega_{ij} p)(x)| \, d \theta \, d \rho;
\]
{\bf note that the $x$ in the above integral lies on the plane $x \cdot \omega=\tau$ and $\rho$, $\theta$ determine a unique $x$ on this plane.}
Hence, using the Cauchy-Schwartz inequality,
\begin{align*}
\tau^2 |p(\tau \omega)|^2
&  \cleq  |P_\tau(\tau, \omega)|^2 +
\left ( \int_\tau^1 \int_0^{2 \pi} 
\frac{\rho}{\sqrt{\rho^2 - \tau^2}} \, d \theta \, d \rho
\right )
\,  \sum_{i<j} \int_\tau^1 \int_0^{2 \pi} 
\frac{\rho}{\sqrt{\rho^2 - \tau^2}} | (\Omega_{ij} p)(x)|^2 \, d \theta \, d \rho
\\
& \cleq  |P_\tau(\tau, \omega)|^2 +
 \sum_{i<j} \int_\tau^1 \int_0^{2 \pi} 
\frac{\rho}{\sqrt{\rho^2 - \tau^2}} | (\Omega_{ij} p)(x)|^2 \, d \theta \, d \rho.
\end{align*}
Hence
\begin{align}
\tau^2 \int_S |p(\tau \omega)|^2 \, d \omega
& \cleq \int_S |P_\tau(\tau, \omega)|^2 \, d \omega
+ \sum_{i<j} \int_S \int_\tau^1 \int_0^{2 \pi}  \frac{\rho}{\sqrt{\rho^2 - \tau^2}} | (\Omega_{ij} p)(x)|^2 \, d \theta \, d \rho \, d \omega.
\label{eq:integPtau}
\end{align}
If we define
\[
f(x) :=   \frac{1}{\sqrt{\rho^2 - \tau^2}} | (\Omega_{ij} p)(x)|^2
\]
then the second integral on the RHS of (\ref{eq:integPtau}) is (below $e = (0,0,1)$)
\begin{align*}
\int_S \int_\tau^1 \int_0^{2 \pi} f(x) \, \rho \, d \theta \, d \rho \, d \omega
& = \int_S \int_{x \cdot \omega = \tau} f(x) \, dS_x \, d \omega
 = \int_S \int_{\R^3} f(x) \, \delta(x \cdot \omega - \tau) \, dx \, d \omega
\\
& = \int_{\R^3} f(x) \left ( \int_S \delta(x \cdot \omega - \tau) \, d \omega \right ) \, dx
\\
& = \int_{\R^3} f(x) \left ( \int_S \delta(|x | \, e \cdot \omega - \tau) \, d \omega \right ) \, dx
\\
& = 2 \pi  \int_{\R^3} f(x) \int_0^ \pi \delta( |x| \, \cos u - \tau) \, \sin u \, du \, dx
\\
&= 2 \pi \int_{\R^3} \frac{ f(x)}{|x|} H(|x| - \tau) \, dx.
\end{align*}
Hence (\ref{eq:integPtau}) gives us, for all $\tau \in [0,1]$,
\begin{align}
\tau^2 \int_S |p(\tau \omega)|^2 \, d \omega
 & \cleq \int_S |P_\tau(\tau, \omega)|^2 \, d \omega
+ \sum_{i<j} \int_{|x| \geq \tau}  \frac{1}{\rho \sqrt{\rho^2 - \tau^2}} | (\Omega_{ij} p)(x)|^2 \, dx
\nn
\\
& \cleq \int_S |P_\tau(\tau, \omega)|^2 \, d \omega
+ \int_\tau^1 \frac{ \rho}{\sqrt{\rho^2 - \tau^2}}  \sum_{i<j} \int_S 
 | (\Omega_{ij} p)(\rho \omega)|^2 \, dw \, d \rho.
\label{eq:Ptaudiff}
\end{align}


\subsection{The proof  of Theorem \ref{thm:backscatter}}

We are given that $\alpha_1(-\omega,\omega, s) = \alpha_2(-\omega,\omega, s)$ for all $\omega \in S$ and all
$s \in [0,2]$. Hence, from Proposition \ref{prop:identity}
\beqn
\int_{x \cdot \omega = \tau} p(x) \, dS_x = - \int_{-1}^\tau \int_{x \cdot \omega=t} p(x) \, k(x,\tau, \omega) \, dS_x \, dt, \qquad \forall \tau \in [-1,0], ~ \forall \omega \in S.
\label{eq:tiden}
\eeqn
It will be more convenient to deal with positive $\tau$ rather than negative $\tau$, so in (\ref{eq:tiden}) we replace $\omega$ by $-\omega$, $\tau$ by $-\tau$ and $t$ by $-t$. We obtain
\beqn
\int_{x \cdot \omega = \tau} p(x) \, dS_x = \int_{\tau}^1 \int_{x \cdot \omega=t} p(x) \, k'(x,\tau, \omega) \, dS_x \, dt, \qquad \forall \tau \in [0,1], ~ \forall \omega \in S
\label{eq:iden}
\eeqn
where
\[
k'(x,\tau, \omega) = -k(x,-\tau, -\omega).
\]
Differentiating (\ref{eq:iden}) with respect to $\tau$, we have, for all $\tau \in [0,1]$ and all $\omega \in S$,
\begin{align*}
P_\tau(\tau, \omega) 
&= - \int_{x \cdot \omega=\tau} p(x) \, k'(x, \tau, \omega) \, dS_x + \int_\tau^1 \int_{x \cdot \omega=t}
p(x) \, k'_\tau(x,\tau,\omega) \, dS_x \, dt;
\end{align*}
here $k'_\tau$ is the $\tau$ partial derivative of $k'$.
Noting that $p$ is supported in the unit ball, we have for all $ \tau \in [0,1]$
\begin{align*}
\int_S |P_\tau(\tau, \omega)|^2 \, d \omega
&
\cleq \int_S  \int_{x \cdot \omega = \tau} |p(x)|^2 \, dS_x \, d \omega
+ \int_\tau^1 \int_S \int_{x \cdot \omega =t}  |p(x)|^2 \, dS_x \, d \omega \, dt
\\
&= \int_B |p(x)|^2 \int_S \delta( x \cdot \omega - \tau) \, d \omega \, dx
+ \int_\tau^1 \int_B |p(x)|^2 \int_S \delta(x \cdot \omega -t) \, d \omega \, dx \, dt.
\end{align*}
These $d \omega$ integrals were computed earlier to be $2 \pi |x|^{-1} H(|x|-\tau)$ and
$2 \pi |x|^{-1} H(|x|-t)$ respectively, so 
\begin{align}
\int_S |P_\tau(\tau, \omega)|^2 \, d \omega
& \cleq \int_{|x| \geq \tau} \frac{|p(x)|^2}{|x|}  \, dx + \int_\tau^1 \int_{|x| \geq t} \frac{|p(x)|^2}{|x|}
 \, dx \, dt
\cleq \int_{|x| \geq \tau} \frac{|p(x)|^2}{|x|} \, dx
\nn
\\
& \cleq \int_\tau^1 \rho \int_S | p(\rho \omega)|^2 \, d \omega.
\label{eq:vol2}
\end{align}
So using (\ref{eq:Ptaudiff}) we obtain, for all $\tau \in [0,1]$ and all $\omega \in S$,
\begin{align*}
\tau^2 \int_S | p(\tau \omega)|^2 \, d \omega
& \cleq \int_\tau^1 \rho \int_S |p(\rho \omega)|^2 \, d \omega + 
\int_\tau^1 \frac{ \rho}{\sqrt{\rho^2 - \tau^2}} \sum_{i<j} \int_S 
 | (\Omega_{ij} p)(\rho \omega)|^2 \, dw \, d \rho.
\end{align*}
If we define
\[
E(\rho) := \int_S |p(\rho \omega)|^2  \, d \omega, \qquad \rho \in [0,1]
\]
then using the angular derivative property  (\ref{eq:angular}) of $p$ we obtain, for all $\tau \in [0,1]$,
\begin{align*}
\tau^2 E(\tau)
&  \cleq  \int_\tau^1 \rho E(\rho) \, d \rho
+ \int_\tau^1 \frac{\rho}{\sqrt{\rho^2 - \tau^2}} E(\rho) \, d \rho
 \cleq \int_\tau^1 \frac{\rho}{\sqrt{\rho^2 - \tau^2}} E(\rho) \, d \rho
\\
& \cleq \int_\tau^1 \frac{ E(\rho)}{\sqrt{\rho-\tau}} \, d \rho.
\end{align*}

Pick any small $\ep>0$; then for all $\tau \in [\ep,1]$ we have
\[
E(\tau) \cleq \int_\tau^1 \frac{E(\rho)}{\sqrt{\rho - \tau}}  \, d \rho.
\]
Substituting this inequality back in itself we obtain, for all $\tau \in [\ep, 1]$, 
\begin{align*}
E(\tau) \cleq \int_{\tau}^1 \int_\rho^1 \frac{ E(s)}{ \sqrt{ \rho - \tau} \, \sqrt{s - \rho}} \, ds \, d \rho
= \int_\tau^1 E(s) \int_\tau^s \frac{ 1}{ \sqrt{ \rho - \tau} \, \sqrt{s - \rho}} \, d \rho \, ds
= \pi \int_\tau^1 E(s) \, ds.
\end{align*}
Hence, by Gronwall's inequality, $E(\tau)=0$ for all $\tau \in [\ep,1]$ for all $\ep>0$. So $p=0$ and the theorem is proved.
%


\section{Proof of Theorem \ref{thm:elementary}}

\begin{enumerate}
\item[(a)] If $p=q_2-q_1$ and $\alpha = \alpha_1 - \alpha_2$ then, from Proposition \ref{prop:identity}, we have for the fixed $\omega$ and all $\tau \in [-1,1]$ that
\[
P(\tau, \omega) := \int_{x \cdot \omega=\tau} p(x) \, dS_x = - \int_{-1}^\tau \int_{x \cdot \omega=t} p(x) \, k(x,\omega,\tau) \, dS_x \, dt.
\]
Since $p=q_2 - q_1 \geq 0$ we obtain
\begin{align*}
P(\tau, \omega) \leq \int_{-1}^\tau \int_{x \cdot \omega =t} p(x) \, |k(x,\omega, \tau)| \, dS_x \, dt
\leq C \int_{-1}^\tau P(\tau, \omega) \, dt, \qquad \forall \tau \in [-1,1].
\end{align*}
Hence by Gronwall's inequality $P(\tau, \omega)=0$ for all $\tau \in [-1,1]$ for this fixed $\omega$. Since $p \geq 0$ and is continuous, this implies $p(x)=0$ on $ x \cdot \omega = \tau$ for all $\tau \in [-1,1]$. Since $p$ is supported in the unit ball we obtain $p=0$.
%
\item[(b)]

Define
\[
k_{max} = \max ~ \{ (|k|+ |k _\tau|)(x, \omega, \tau)  \, : \,  x \in \R^3, ~ \omega \in S, ~ \tau \in [-1,0],  - 1 \leq x \cdot \omega \leq \tau, ~ |x| \leq 1\}.
\]
Again, from the hypothesis and Proposition  \ref{prop:identity} we have
\[
P(\tau, \omega) = - \int_{-1}^\tau \int_{x \cdot \omega=t} p(x) \, k(x,\omega,\tau) \, dS_x \, dt,
\qquad \forall \omega \in S, ~ \tau \in [-1,0].
\]
Hence, for all $\omega \in S, ~ \tau \in [-1,0]$ we have
\begin{align*}
|P_\tau( \tau, \omega)|  & \leq  \int_{x \cdot \omega=\tau} |p(x)| \, |k(x,\omega,\tau)| \, dS_x 
+ \int_{-1}^\tau \int_{x \cdot \omega=t} |p(x)| \, |k_\tau(x,\omega,\tau)| \, dS_x \, dt
\\
& \leq k_{max} \left ( \int_{x \cdot \omega=\tau} |p(x)| \, dS_x 
+ \int_{-1}^\tau \int_{x \cdot \omega=t} |p(x)| \, dS_x \, dt \right )
\end{align*}
so, since $p$ is supported in the unit ball, from the Cauchy-Schwartz inequality
\[
|P_\tau(\omega, \tau)|^2 \leq 3 \pi k_{max}^2 \left ( \int_{x \cdot \omega=\tau} |p(x)|^2 \, dS_x 
+ \int_{-1}^\tau \int_{x \cdot \omega=t} |p(x)|^2 \, dS_x \, dt \right )
\]
and hence
\begin{align*}
\int_{-1}^ 1 |P_\tau(\omega, \tau)|^2 \, d \tau
& \leq 3 \pi k_{max}^2 \left ( \int_{-1}^1 \int_{x \cdot \omega=\tau} |p(x)|^2 \, dS_x \, d \tau
+ \int_{-1}^1 \int_{-1}^\tau \int_{x \cdot \omega=t} |p(x)|^2 \, dS_x \, dt \, d \tau \right )
\\
& \leq 12 \pi k_{max}^2 \int_{\R^3} |p(x)|^2 \, dx.
\end{align*}

Noting that $P(\tau, \cdot)=0$ for $|\tau| \geq 1$, from the Plancherel formula and the observation that
$P(-\tau,\omega) = P(\tau, - \omega)$ we have
\begin{align*}
\int_{\R^3} |p(x)|^2 \, dx &= \frac{1}{8 \pi^2} \int_{-1}^1 \int_S |P_\tau (\omega, \tau)|^2 \, d \omega \, d \tau   
= \frac{1}{4 \pi^2} \int_{-1}^0 \int_S |P_\tau (\omega, \tau)|^2 \, d \omega \, d \tau 
\\
& \leq 3 k_{max}^2 \int_{\R^3} |p(x)|^2 \, dx. 
\end{align*}
Hence $p=0$ if we can find an $M>0$ so that $3 k_{max}^2 < 1$ if $\|q\|_{C^2(\R^3)} \leq M$.

The expression for $k(x,\omega, \tau)$ is given in Proposition (\ref{prop:identity}). With that in mind, we note that as $x, \tau$ vary over the region region $-1 \leq x \cdot \omega \leq \tau \leq 0$,
the point $(x, s)$ with $s \in [x \cdot \omega, 2 \tau  - x \cdot \omega]$ will vary over the region
$(x,t)$ with $-1 \leq x \cdot \omega \leq t \leq 0$. Hence, if we define
\[
\|u\|_* := \max \{ |u(x,t)|+|u_t(x,t)| \, : \, |x| \leq 1, ~ -1 \leq x \cdot \omega \leq t \leq 0 \},
\]
where $u(x,t)$ is the solution of (\ref{eq:ude}) - (\ref{eq:uic}) corresponding to $q=q_1$ or $q=q_2$,
then from the expression for $k(x,\omega,\tau)$ in Proposition (\ref{prop:identity}) we have
\[
k_{max} \leq  4 \|u\|_* + 2 \|u\|_*^2 + 8 \|u\|_* + 4\|u\|_*^2 + 4 \|u\|_*^2  = 12 \|u\|_* + 10 \|u\|_*^2
\]
So the proof of the proposition will be complete if we can show the following.

\begin{proposition}\label{prop:stability}
If $q(x)$ is a smooth function on $\R^3$ with support in the unit ball, $u(x,t, \omega)$ the solution of (\ref{eq:ude}) - (\ref{eq:uic}), and $\|q\|_{C^2}$ is small enough (independent of $u$) then
$ \|u\|_* \leq 8 \|q\|_{C^2}$.
\end{proposition}
The proof of Proposition \ref{prop:stability} is given in section \ref{sec:stability}.
\end{enumerate}

\subsection{Proof of Proposition \ref{prop:stability}}\label{sec:stability}

$u(x,t,\omega)$ is the solution of the characteristic initial value problem (\ref{eq:ude}) - (\ref{eq:uic}). Further, $u_t(x,t,\omega)$ is also the solution of (\ref{eq:ude}) - (\ref{eq:uic}) except that the characteristic condition (\ref{eq:ubc}) will  have a different RHS. So Proposition \ref{prop:stability} will follow from an estimate for a characteristic initial value problem if we can just determine
$u_t(x, x \cdot \omega, \omega)$.  

There is no loss of generality in assuming that $\omega= (0,0,1)$; below $u_i$ will denote the partial derivative of $u$ with respect to 
$x_i$ and we will stop showing the dependence of $u$ on $\omega$. Since
\[
u(x_1, x_2, x_3, x_3) = -\frac{1}{2} \int_{-\infty}^0 q( x_1, x_2, x_3 + s) \, ds, 
\]
we have
\[
(u_3 + u_t)(x, x_3) = -\frac{q(x)}{2}.
\]
Also, from (\ref{eq:ude})
\begin{align*}
\pa_3[ ( u_t - u_3)(x, x_3)] & = (u_{tt} - u_{33})(x,x_3) = (u_{11} + u_{22} - qu)(x, x_3)
\\
&=  -\frac{1}{2} \int_{-\infty}^0 (q_{11} + q_{22})( x_1, x_2, x_3 + s) \, ds
+  \frac{q(x)}{2} \int_{-\infty}^0 q( x_1, x_2, x_3 + s) \, ds
\\
&= (\text{call it})\, Q(x).
\end{align*}
Hence
\[
(u_t - u_3)(x, x_3) = \int_{-\infty}^0 Q(x_1, x_2, x_3 + s ) \, ds
\]
so
\[
u_t(x, x_3) =  -\frac{ q(x)}{4} + \frac{1}{2} \int_{-\infty}^0 Q(x_1, x_2, x_3 + s ) \, ds.
\]
Since $Q$ depends on the second order derivatives of $q$, Proposition \ref{prop:stability} follows from the following result for solutions of characteristic initial boundary value problems.
\begin{figure}
\begin{center}
\epsfig{file=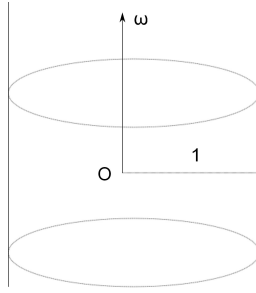, height=1.5in}
\end{center}
\caption{Support of $f$}
\label{fig:cylinder}
\end{figure}
Suppose $\omega$ is a unit vector in $\R^3$,  $q(x)$  a smooth function on $\R^3$ which is supported on the unit ball. 
Further, let $f(x)$ be a smooth function on $\R^3$ with $f(x)$ supported in the cylinder of radius $1$ with
 axis the line through the origin parallel to $\omega$, and $f(x)$ zero if $ x \cdot \omega \leq -1$, that is 
\[
\text{supp} ~ f \subseteq \{ x \in \R^3 \, : \, \| x -  \la x, \omega \ra \omega \| \leq 1, ~ x \cdot \omega \geq -1 \};
\]
see Figure \ref{fig:cylinder}.
Let $a(x,t)$ be the solution of the characteristic IVP
\begin{align}
a_{tt} - \Delta a + qa &= 0, \qquad (x,t) \in \R^3 \times \R, ~~t \geq x \cdot \omega
\label{eq:ade}
\\
a(x, x \cdot \omega) &= f(x), ~~ x \in \R^3,
\label{eq:abc}
\\
a(x,t)&=0 \qquad t < -1.
\label{eq:aic}
\end{align}
Define $\|a\|_\infty := \sup \{ |a(x,t)| \, : \, \|x\| \leq 1, ~ -1 \leq x \cdot \omega \leq t \leq 0 \}$,
\[
\|q\|_\infty := \sup_{x \in \R^3} |q(x)|, \qquad
\|f\|_* := \sup \{ | (\omega \cdot \nabla f)(x)| \, : \, x \in \R^3, ~ x \cdot \omega \leq 0 \}.
\]
We show that if $\|q\|_\infty \leq 1/4$  then
\[
\|a\|_\infty \leq 2  \|f\|_*.
\]
We obtain a very crude estimate but that will be enough for our purposes. A much sharper estimate may be obtained along with a proof of the existence of $a$ using a Volterra argument as in [Ro74].

We prove our claim by expressing $a$ as the solution of an integral equation. The derivation of this integral equation is formal, using 
the Green's function for the wave equation; a rigorous derivation would imitate the construction of the Green's function for the wave 
equation. For a fixed $(x,t) \in \R^3 \times \R$ with $ t > x \cdot \omega \geq -1$, define
\[
G(y,s) = \frac{1}{4 \pi} \frac{ \delta( t-s- |x-y| )}{ |x-y|};
\]
then $G(y,s)$ is the solution of the  backward IVP
\begin{align*}
\Box_{y,s} G(y, s) &= \delta(x-y,t-s), \qquad (y,s) \in \R^3 \times \R
\\
G(y, s) &= 0, ~~ s>t, ~ y \in \R^3.
\end{align*}
Since, for a fixed $y$, $G(y,s)$ is zero for $s$ large, and, for a fixed $s$, $G(y,s)$ is compactly supported in $y$, an application of the divergence theorem gives us 
\begin{align}
a(x,t) & = \int_{s \geq y \cdot \omega} a(y,s) \, \delta(x-y, t-s) \, dy \, ds
= \int_{s \geq y \cdot \omega} a(y,s) \, \Box_{y,s} G(y, s)  \, dy \, ds
\nn
\\
& = -\int_{s \geq y \cdot \omega} \Box_{y,s} a \, \, G  \, dy \, ds
+ 
\int_{s \geq y \cdot \omega} ( a G_s - G a_s)_s + \nabla_y \cdot( G \nabla_y a- a \nabla_y G) \, dy \, ds
\nn
\\
& = -\int_{s \geq y \cdot \omega} qa \, G \, dy \, ds
+ \frac{1}{\sqrt{2}} \int_{y \cdot \omega = s} G (a_s + \omega \cdot \nabla a)
- a( G_s + \omega \cdot \nabla G) \, dS_{y,s}
\nn
\\
& = - \int_{s \geq y \cdot \omega} qa \, G  \, dy \, ds
+ \int_{\R^3}  G(y, \omega \cdot y) \, \, \omega \cdot \nabla_y ( a (y, \omega \cdot y) )
- a(y, \omega \cdot y) \, \, \omega \cdot \nabla_y ( G (y, \omega \cdot y) )
\, dy
\nn
\\
& = -\int_{s \geq y \cdot \omega} qa \, G  \, dy \, ds
+ 2 \int_{\R^3}  G(y, \omega \cdot y) \, \, \omega \cdot \nabla_y ( a (y, \omega \cdot y) )
\, dy
\nn
\\
& = - \int_{s \geq y \cdot \omega} qa \, G  \, dy \, ds
+ 2 \int_{\R^3}  G(y, \omega \cdot y) \, \, \omega \cdot \nabla_y f(y) 
\, dy.
\label{eq:tempt}
\end{align}
In the last step we used the divergence theorem on the plane $s= y \cdot \omega$ and note that
$a(y,y \cdot \omega)$ is zero if $y \cdot \omega \leq -1$ and $G(y, y \cdot \omega)=0$ if
$ y \cdot \omega \geq t$.

Let $g(y) := \omega \cdot \nabla_y f(y)$ and extend $a(y,s)$ to be zero for $ s < y  \cdot \omega$;
then from (\ref{eq:tempt})
\begin{align*}
4 \pi a(x,t) &=  2 \int_{\R^3}  \frac{g(y) \, \delta (t- y \cdot \omega - |x-y|) } {|x-y|}
\, dy 
- \int_\R \int_{\R^3}  \frac{ q(y) \, a(y,s) \,
\delta (t-s - |x-y|)}{ |x-y|}  \, dy \, ds
\\
& = 2\int_{\R^3}  \frac{g(y)}{ |x-y|} \,
\delta (t- y\cdot \omega - |x-y|) \, dy 
-  \int_{\R^3}  \frac{q(y) \, a(y, t - |x-y|)}{ |x-y|} \, dy 
\end{align*}
For the rest of this subsection we will assume that $|x| \leq 1$, $-1 \leq x \cdot \omega < t \leq 0$.

For the first integral uses values of $g$ on the set 
$y \cdot \omega + |x-y| \leq t$ so $y \cdot \omega \leq 0$; also, because of the support of $a$, the second integral uses values of $a(y,s)$ on the region $y \cdot \omega \leq s$ and $s \leq 0$ because $t-|x-y| \leq t \leq 0$. So 
$|g(y)| \leq \|f\|_*$ and $|a(y,s)| \leq \|a\|_\infty$ on the region of integration. 
Hence because of the support of $g$ and $q$ we have
\begin{align*}
4 \pi |a(x,t)|  & \leq 2 \|f\|_* \int_{y \cdot \omega \geq -1}  \frac{\delta( t - y \cdot \omega - |x-y|)}{ |x-y|} \, dy 
+ \|q\|_\infty \, \|a\|_\infty   \int_{|y| \leq 1}  \frac{1}{ |x-y|} \, dy 
\\
& \leq 2 \|f\|_* \int_{(x+y) \cdot \omega \geq -1}  \frac{\delta( t - (x+y) \cdot \omega - |y|)}{ |y|} \, dy 
+ \|q\|_\infty \, \|a\|_\infty   \int_{|y| \leq 2}  \frac{1}{ |y|} \, dy .
\end{align*}

WLOG we assume that $\omega=(0,0,1)$; then the first integral is  over the paraboloid
\[
\Sigma : = \{ y \in \R^3 :  - \left ( y_3 - \frac{\lambda}{2} \right ) 
= \frac{1}{2 \lambda} (y_1^2 + y_2^2), ~ x_3 + y_3 \geq -1  \}
\] 
where $\lambda := t-x_3>0$. On $\Sigma$, the variables $y_1, y_2$ are restricted to the disk
\[
1 + x_3 + \frac{\lambda}{2} \geq \frac{1}{2 \lambda} (y_1^2 + y_2^2)
\]
that is where 
$
y_1^2 + y_2^2 \leq R^2 
$
with $R^2 = \lambda( \lambda + 2(x_3+1) )$. 
Also, on $\Sigma$
\[
\left | \frac{\pa}{\pa y_3}  ( t - x_3 - y_3 - |y|) \right | =  \left | 1 + \frac{y_3}{|y|} \right |
= \frac{ |y_3 + |y||}{|y|} = \frac{t-x_3}{|y|} = \frac{\lambda}{|y|}
\]
so
\begin{align*}
 \int_{x_3 + y_3 \geq -1}   \frac{\delta (t- x_3 - y_3 - |y|)}{ |y|} \,
 \, dy 
& \leq \frac{1}{\lambda} \int_{y_1^2 + y_2^2 \leq R^2}\, dy_1 \, dy_2 
= \pi ( \lambda + 2 (x_3 + 1) ) \|f\|_*.
\end{align*}
Hence
\begin{align*}
|a(x,t)| & \leq \frac{ \lambda + 2(x_3+1)}{2} \, \|f\|_* + 2 \|q\|_\infty \, \|a\|_\infty
\\
& \leq 2 \|f\|_* + \|q\|_\infty \, \|a\|_\infty,
\end{align*}
so
\[
\|a\|_\infty \leq  \|f\|_* + 2 \|q\|_\infty \, \|a\|_\infty
\]
so if $\|q\|_\infty \leq 1/4$ then
\[
\|a\|_\infty \leq 2 \|f\|_*.
\]

\section{Appendix}
\subsection{Linearization about $q=0$}\label{subsec:born}
We calculate the formal derivative of the map $q \to \alpha(\theta,\omega, s)$ at $q=0$. We show that if $\theta \neq \omega$ then 
this formal derivative at $q=0$ is
\[
p(x) \to \frac{-1}{ 4 \pi \, |\theta - \omega|} \int_{ x \cdot(\theta - \omega)=s} p(x) \, dS_x
\]
and is
\[
p(x) \to \frac{-\delta(s)}{4 \pi}  \, \int_{\R^3} p(x) \, dx
\]
when $ \theta = \omega$. This matches what has been obtained in the literature, from the linearization about $q=0$, of the 
frequency domain far field patterns.

Let $p(x)$ be a smooth function on $\R^3$ which is supported in the unit ball $B$. Let $v(x,t,\omega)$ be the solution of the IVP
\begin{align}
v_{tt} - \Delta v &= - p(x) \, \delta(t- x \cdot \omega),  \qquad (x,t) \in \R^3 \times \R
\label{eq:vappde}
\\
v(x,t) &=0, \qquad x \in \R^3, ~~ t < -1.
\label{eq:vappic}
\end{align}
Then the formal derivative of the map $q \to \alpha(\theta,\omega, s)$ at $q=0$ is the map sending $p(x)$ to 
$-\frac{1}{2 \pi} \int_{x \cdot \theta=1} (\theta \cdot \nabla v)(x, 1-s, \omega) \, dS_x$, which is equal to
$\frac{1}{2 \pi} \int_{x \cdot \theta=1} v_t(x, 1-s, \omega) \, dS_x$ by an argument identical to the one used for proving Theorem
\ref{thm:forward}c.

Now
\begin{align*}
\frac{1}{2 \pi} \int_{x \cdot \theta=1} v_t(x, 1-s, \omega) \, dS_x
 & = - \frac{1}{8 \pi^2} \int_{x \cdot \theta=1} \int_{\R^3 \times \R} p(y) \, \delta( \sigma - y \cdot \omega)
\, \frac{ \delta'(1-s - \sigma - |x-y|)}{|x-y|} \, dy \, d \sigma \, dS_x
\\
& =  - \frac{1}{8 \pi^2} \int_{x \cdot \theta=1} \int_{\R^3 } p(y) \, 
\, \frac{ \delta'(1-s -y \cdot \omega  - |x-y|)}{|x-y|} \, dy \, dS_x.
\end{align*}
So, if $\phi(s)$ is a smooth, compactly supported function on $\R$, then
\begin{align*}
\frac{1}{2 \pi} \int_\R \int_{x \cdot \theta=1} v_t(x, 1-s, \omega) \, \phi(s) \, dS_x \, ds
& =  - \frac{1}{8 \pi^2}  \int_\R \int_{x \cdot \theta=1} \int_{\R^3 } p(y) \, 
\, \frac{ \delta'(1-s -y \cdot \omega  - |x-y|)}{|x-y|} \, \phi(s) \,  dy \, dS_x \, ds
\\
& =  - \frac{1}{8 \pi^2}  \int_{x \cdot \theta=1} \int_{\R^3 } p(y) \, 
\, \frac{ \phi'(1 -y \cdot \omega  - |x-y|)}{|x-y|}  \,  dy \, dS_x 
\\
& =  - \frac{1}{8 \pi^2}  \int_{\R^3} p(y) \, \int_{x \cdot \theta=1}  \frac{ \phi'(1 -y \cdot \omega  - |x-y|)}{|x-y|}  \, dS_x \, dy.
\end{align*}
We show below that for $|y|<1$ we have
\beqn
 \int_{x \cdot \theta=1}  \frac{ \phi'(1 -y \cdot \omega  - |x-y|)}{|x-y|}  \, dS_x =   2 \pi 
 \, \phi( y \cdot (\theta - \omega) )
 =  2 \pi \, \int_\R \phi(s) \, \delta( s - y \cdot (\theta - \omega) ) \, ds
 \label{eq:appiden}
\eeqn
which will prove that
\begin{align*}
\frac{1}{2 \pi}  \int_{x \cdot \theta=1} v_t(x, 1-s, \omega)  \, dS_x 
& = 
\frac{-1}{4 \pi} \int_{\R^3} p(y) \, \delta( s - y \cdot (\theta - \omega) ) \, dy
\\
& = 
\begin{cases}
\frac{-1}{4 \pi \, |\theta - \omega|} \int_{y \cdot(\theta - \omega) = s} p(y) \, dS_y,  & \theta \neq \omega, 
\\
\frac{-\delta(s)}{4 \pi} \, \int_{\R^3} p(y) \, dy, & \theta = \omega,
\end{cases}
\end{align*}
which proves our claim.

It remains to prove (\ref{eq:appiden}). For $|y|<1$, using a translation, a rotation and polar coordinates (below $r= \sqrt{x_1^2+x_2^2}$) we have
\begin{align*}
 \int_{x \cdot \theta=1}  \frac{ \phi'(1 -y \cdot \omega  - |x-y|)}{|x-y|} & \, dS_x
  =  \int_{x \cdot \theta=1-y \cdot \theta}  \frac{ \phi'(1 -y \cdot \omega  - |x|)}{|x|}  \, dS_x
 \\
& = \int_{x_3=1-y \cdot \theta}  \frac{ \phi'(1 -y \cdot \omega  - |x|)}{|x|}  \, dS_x
 \\
 & = 2 \pi \int_0^\infty \frac{ \phi'( 1 - y \cdot \omega - \sqrt{ r^2 + (1- y \cdot \theta)^2} )} { \sqrt{ r^2 + (1- y \cdot \theta)^2}}
 \; r \; dr
 \\
 & =  2 \pi \, \int_{-\infty}^{y \cdot (\theta-\omega)} \phi'(t) \, dt
 \\
 & =  2 \pi \, \phi( y \cdot (\theta - \omega) ).
\end{align*}

\subsection{Far field patterns for translated potentials}\label{subsec:translate}
If $\alpha(\theta,\omega, s)$ is the far field pattern for $q(x)$ and $\beta(\theta,\omega,s)$ the far field pattern of its translate $q(x+a)$,  $a \in \R^3$, then we show that
\[
\beta(\theta,\omega,s) = \alpha(\theta, \omega, s + a \cdot(\theta-\omega)).
\]

From Theorem \ref{thm:forward}(c),
the far field pattern for $q$ is
\[
\alpha(\theta, \omega, s) = \frac{1}{2 \pi} \int_{x \cdot \theta =1} u_t(x, 1-s, \omega) \, dS_x
= \frac{1}{2 \pi} \int_{ x \cdot \theta = \tau} u_t(x, \tau -s, \omega) \, dS_x,
\]
for any large enough $\tau$,
where $u(x,t,\omega)$ is the solution of the IVP (\ref{eq:utempde}), (\ref{eq:utempic}).

Let $v(x,t,\omega)$ be the solution of
\begin{align*}
(v_{tt} - \Delta v)(x,t) + q(x+a) v(x,t)& = q(x+a) \delta(t- x \cdot \omega), \qquad (x,t) \in \R^3 \times \R,
\\
v(x,t,\omega)&=0, \qquad t <<0;
\end{align*}
then
\begin{align*}
(v_{tt} - \Delta v)(x-a,t) + q(x) v(x-a,t) &= q(x) \delta(t+ a \cdot \omega - x \cdot \omega), \qquad (x,t) \in \R^3 \times \R,
\\
v(x,t,\omega)&=0, \qquad t <<0
\end{align*}
and hence $v(x-a,t, \omega) = u(x, t + a \cdot \omega, \omega)$, so 
\[
v(x,t, \omega) = u(x+a, t + a \cdot \omega, \omega), \qquad (x,t) \in \R^3 \times \R.
\]
The far field pattern for $q(x+a)$ is (for $\tau$ large enough)
\begin{align*}
\beta(\theta, \omega, s) & = \frac{1}{2 \pi}\int_{x \cdot \theta = \tau} v_t(x, \tau-s, \omega) \, dS_x
\\
& =  \frac{1}{2 \pi}\int_{x \cdot \theta = \tau} u_t(x+a, \tau-s+ a \cdot \omega, \omega) \, dS_x
\\
& =  \frac{1}{2 \pi}\int_{x \cdot \theta = \tau + a \cdot \theta} u_t(x, \tau-s+ a \cdot \omega, \omega) \, dS_x
\\
& = \frac{1}{2 \pi}\int_{x \cdot \theta = \tau} u_t(x, \tau-s+ a \cdot (\omega-\theta), \omega) \, dS_x,
\qquad \text{from Theorem \ref{thm:forward}(c)}
\\
& = \alpha(\theta, \omega, s + a(\theta - \omega)).
\end{align*}

\section{Acknowledgments}
Rakesh's work was partially supported by NSF grants DMS 0907909, DMS 1312708 and  Gunther Uhlmann's work was partially 
supported by the NSF and a Simons Fellowship and
 part of this work was done when he was an Ordway Distinguished Visitor at the University of Minnesota.
 %

\end{document}